\documentclass[a4paper]{amsart}

\usepackage{pstricks-add}
\usepackage{float}

\usepackage[colorlinks=true]{hyperref}

\newtheorem{thm}{Theorem}[section]

\newtheorem{prop}[thm]{Proposition}
\newtheorem{lem}[thm]{Lemma}
\newtheorem{cor}[thm]{Corollary}

\theoremstyle{definition}
\newtheorem{defn}[thm]{Definition}
\newtheorem{ex}[thm]{Example}

\theoremstyle{remark}
\newtheorem{rem}[thm]{Remark}

\newtheorem*{note}{Note}        

\newcommand{\smc}{\mbox{\,\tiny{$\circ $}\,}}         

 \def\bb#1#2{\left\{#1,#2\right\}}

 \def\brk#1#2{\left[#1,#2\right]}
 
 \def\SNbrk#1#2{\left[#1,#2\right]_{{}_{SN}}}
 \def\FNbrk#1#2{\left[#1,#2\right]_{{}_{FN}}}

\begin{document}
\title{Hyperstructures on Lie algebroids}

\author{P. Antunes}
\address{CMUC, Department of Mathematics, University of Coimbra, 3001-454 Coimbra, Portugal}
\email{pantunes@mat.uc.pt}
\author{J.M. Nunes da Costa}
\address{CMUC, Department of Mathematics, University of Coimbra, 3001-454 Coimbra, Portugal}
\email{jmcosta@mat.uc.pt}

\begin{abstract}
We define hypersymplectic structures on Lie algebroids recovering, as particular cases, all the classical results and examples of hypersymplectic structures on manifolds. We prove a $1$-$1$ correspondence theorem between hypersymplectic structures and (pseudo-)hyperk\"{a}hler structures. We show that the hypersymplectic framework is  very rich in already known compatible pairs of tensors
such as Poisson-Nijenhuis, $\Omega N$ and $P \Omega$ structures.
\end{abstract}

\maketitle

\


\

\section{Introduction}             

Hypersymplectic structures on manifolds were introduced by Xu in \cite{Xu97}.
Aiming to study  hyperk\"{a}hler structures on manifolds from the viewpoint of symplectic geometry, Xu was led to introduce, in a natural way, the notion of hypersymplectic structure. The aim of this paper is to define and study hypersymplectic structures on Lie algebroids. Our definition is inspired  in \cite{Xu97} and, generalizing a result from \cite{Xu97} on manifolds, we prove that there exists a $1$-$1$ correspondence between hypersymplectic and hyperk\"{a}hler structures on a Lie algebroid. We also show that hypersymplectic structures on a Lie algebroid provide other type of interesting structures on the Lie algebroid, such as Poisson-Nijenhuis, $\Omega N$ and $P \Omega$ structures. The results of this paper are from \cite{A10} and were never published. We want to stress that the proofs we give here are different from those in \cite{A10}. We have improved the techniques and the proofs became much more elegant.

It is worth noticing that Hitchin \cite{hitchin} called hypersymplectic structures on a manifold $M$ to what we call, in this work, para-hypersymplectic structures on the Lie algebroid $TM$. On the other hand, the structures which we call hypersymplectic are in $1$-$1$ correspondence with hyperk\"{a}hler structures (Theorem~\ref{Thm_1-1_correspondence}). For this reason, many authors refer to them simply as hyperk\"{a}hler structures.

The paper is divided into six sections. Since our computations widely use the big bracket -- the Poisson bracket induced by
the symplectic structure on the cotangent bundle of a supermanifold, Section 1 contains a short review of
Lie algebroids in the supergeometric framework and their deformation by Nijenhuis tensors and bivector fields. We also review the definition of  Schouten-Nijenhuis bracket of multivectors on a Lie algebroid $A$ as a derived bracket expression, and give a definition of the Fr\"{o}licher-Nijenhuis bracket of two $A$-valued forms in supergeometric terms. In Section 2 we introduce the concept of {\large$\boldsymbol{\varepsilon}$}-hypersymplectic structure on a Lie algebroid and we study the properties of the tensors  induced on the Lie algebroid by this structure.
One of the interesting features of {\large$\boldsymbol{\varepsilon}$}-hypersymplectic structures, discussed in Section 3, is that pairs of these induced tensors determine  well known structures on the Lie algebroid, such as $P \Omega$, $\Omega N$ and Poisson-Nijenhuis structures. In Section 4, we consider a particular case, that amounts to fix a sign in the {\large$\boldsymbol{\varepsilon}$}-hypersymplectic structure, and we show that the number of $P \Omega$, $\Omega N$ and Poisson-Nijenhuis structures on the Lie algebroid increases, when compared to the general case. In Section 5 we take the opposite sign, and consider two cases: hypersymplectic and para-hypersymplectic structures. In both cases we are able to define a pseudo-metric on the Lie algebroid. Using this pseudo-metric, we prove a $1$-$1$ correspondence between (para-)hypersymplectic  and (para-)hyperk\"{a}hler structures on the Lie algebroid. The paper closes with
an example in  $T\mathbb{R}^4$ that provides many (para-)hypersymplectic structures.

\section{Preliminaries on Lie algebroids}             

\subsection{Lie algebroids in supergeometric terms}  \label{subsection:1.1}
We begin this section by introducing the supergeometric formalism, following the same approach as in \cite{voronov,roy}. Given a vector bundle $A \to M$, we denote by $A[n]$ the graded manifold obtained by shifting the fibre degree by $n$. The graded manifold $T^*[2]A[1]$ is equipped with a canonical symplectic structure which induces a Poisson bracket on its algebra of functions $\mathcal{F}:=C^\infty(T^*[2]A[1])$. This Poisson bracket is sometimes called the \emph{big bracket} (see \cite{YKS92}).

Let us describe locally this Poisson algebra. Fix local coordinates $x^i, p_i,\xi^a, \theta_a$, $i \in \{1,\dots,n\}, a \in \{1,\dots,d\}$, on $T^*[2]A[1]$, where $x^i,\xi^a$ are local coordinates on $A[1]$ and $p_i, \theta_a$ are their associated moment coordinates. In these local coordinates, the Poisson bracket is given by
 $$ \{p_i,x^i\}=\{\theta_a,\xi^a\}=1,  \quad  i =1, \dots, n, \, \, a=1, \dots , d, $$
while all the remaining brackets vanish.

The Poisson algebra of functions $\mathcal{F}$ is endowed with an $(\mathbb{N} \times \mathbb{N})$-valued bidegree. We
define this bidegree (locally but it is well defined globally, see \cite{voronov, roy}) as follows: the coordinates on the base
manifold $M$, $x^i$, $i \in \{1,\dots,n\}$, have bidegree $(0,0)$, while the coordinates on the fibres, $\xi^a$, $a \in \{1,\dots,d\}$,
have bidegree $(0,1)$ and their associated moment coordinates, $p_i$ and $\theta_a$, have bidegrees $(1,1)$ and $(1,0)$, respectively.
We denote by $\mathcal{F}^{k,l}$ the $C^\infty(M)$-module of functions of bidegree $(k,l)$ and we verify that the big bracket has bidegree $(-1,-1)$, i.e.,
$$\{\mathcal{F}^{k_1,l_1},\mathcal{F}^{k_2,l_2}\}\subset \mathcal{F}^{k_1+k_2-1,l_1+l_2-1}.$$

\

Let us recall that a \emph{Lie algebroid} structure on a vector bundle $A \to M$  is a pair $(\rho, [.,.])$ where
\begin{itemize}
  \item $\rho: A\longrightarrow TM$ is a morphism of vector bundles, called the \textit{anchor},
  \item $\brk..$ is a Lie bracket on the space of sections $\Gamma(A)$
\end{itemize}
satisfying the \emph{Leibniz rule}
$$\brk{X}{fY}=f\brk{X}{Y} + (\rho(X) \cdot f)\, Y,$$
for all $f\in C^\infty(M)$ and \mbox{$X,Y \in \Gamma(A)$}.

\begin{thm}[\cite{roythesis,vaintrob}]
There is a $1$-$1$ correspondence between Lie algebroid structures on $A \to M$ and functions $\mu \in \mathcal{F}^{1,2}$ such that \mbox{$\{\mu,\mu\}=0$}.
\end{thm}

The anchor and bracket associated to a given $\mu\in \mathcal{F}^{1,2}$ are defined, for all $X,Y \in \Gamma(A)$ and $f \in C^\infty(M)$, by the derived bracket expressions
\begin{equation}\label{def_rho_and_brk_in_bb}
    \rho(X)\cdot f=\{\{X,\mu\},f\} \quad {\hbox{and}} \quad {[X,Y]=\{\{X,\mu\},Y\}}.
\end{equation}
Moreover, the differential of the Lie algebroid is given by ${\rm d}(\sigma)=\bb{\mu}{\sigma}$, for all $\sigma \in \Gamma(\bigwedge^\bullet A^*)$.

When using the supergeometric formalism, we shall denote a Lie algebroid by the pair $(A, \mu)$ instead of the triple $(A,\rho, [.,.])$.

\subsection{Deformation of Lie algebroids}          

Let $(A, \rho, [.,.])$ be a Lie algebroid and $\mu$ the function in $\mathcal{F}^{1,2}$ corresponding to the Lie algebroid structure $(\rho, [.,.])$ on $A$.

Let $N\in \Gamma(A\otimes A^*)$ be a $(1,1)$-tensor seen as a vector bundle morphism $N:A \to A$.
The deformation of the Lie bracket $[.,.]$ by \mbox{$N$} is defined, for all $X,Y \in \Gamma(A)$, by
$$[X,Y]_N =[NX,Y]+[X,NY]-N[X,Y].$$
The deformed structure $(\rho \smc N, [.,.]_N)$ is given, in supergeometric terms, by $\mu_{N}:=\{N,\mu\}\in\mathcal{F}^{1,2}$.
The deformation of $\mu_N$ by a $(1,1)$-tensor $S\in \Gamma(A\otimes A^*)$ is denoted by
$\mu_{N,S}$, i.e.
\begin{equation}\label{notation_mu_NS}
    \mu_{N,S}=\{S,\{N, \mu \}\},
\end{equation}
while the deformed bracket associated to $\mu_{N,S}$  is denoted by $[.,.]_{N,S}$.

The Nijenhuis torsion of $N$, $\mathcal{T}N$, is given, for all $X,Y \in \Gamma (A)$, by

\begin{equation*} \label{torsion}
\mathcal{T}N (X,Y)= [NX,NY]-N[X,Y]_{N}.
\end{equation*}

A vector bundle endomorphism $N:A \to A$ is a {\em Nijenhuis} tensor on a Lie algebroid $(A, \mu)$ if its Nijenhuis torsion vanishes. In this case the deformed bracket $[.,.]_N$ is a Lie bracket and $(A, \mu_N)$ is a Lie algebroid.

When $N^2= - {\rm Id}_A$ (resp. $N^2= {\rm Id}_A$), $N$ is said to be an \emph{almost complex structure} (resp. \emph{almost para-complex structure}). If moreover $\mathcal{T}N=0$, then we can remove the prefix ``almost'' and $N$ is a \emph{complex structure} (resp. \emph{para-complex structure}\footnote{In other works (including some of ours) these structures are called \emph{(almost) product} instead of \emph{(almost) para-complex}. We use the prefix ``para'' to express the fact that some sign is switched in comparison with the classical notion (without prefix). This change of terminology, when compared to previous works, enables us to express in an unified way the properties that are satisfied by both the classical structure and its para-version.}).

\

The deformation of the Lie bracket $[.,.]$ by a bivector $\pi \in \Gamma(\bigwedge^2 A)$ is a bracket on $\Gamma(A^*)$ defined, for all $\alpha, \beta \in \Gamma(A^*)$, by
$$[\alpha, \beta]_{\pi}=\mathcal{L}_{\pi^{\sharp}(\alpha)}\beta - \mathcal{L}_{\pi^{\sharp}(\beta)}\alpha - {\rm d}(\pi(\alpha, \beta)),$$
where ${\rm d}$ is the differential of the Lie algebroid $(A, \mu)$, $\mathcal{L}$ is the Lie derivative determined by ${\rm d}$ and $\pi^\sharp: \Gamma(A^*)\to \Gamma(A)$ is defined by setting $\langle\beta,\pi^\sharp(\alpha)\rangle:=\pi(\alpha,\beta)$, for all $\alpha, \beta \in \Gamma(A^*)$.

The deformed structure $(\rho \smc \pi^\sharp, [.,.]_{\pi})$ corresponds to $\mu_{\pi}:=\bb{\pi}{\mu} \in \mathcal{F}^{1,2}$. Moreover, if $\pi$ is a Poisson bivector on $A$, $[.,.]_{\pi}$ is a Lie bracket and $(A^*, \mu_{\pi})$ is a Lie algebroid.

Consider a $(1,1)$-tensor $N\in \Gamma(A\otimes A^*)$ and a bivector $\pi \in \Gamma(\bigwedge^2 A)$. The Magri-Morosi concomitant $C_{\pi, N}$ of $\pi$ and $N$~\cite{magrimorosi} is the diference between the deformed brackets on $\Gamma(A^*)$, $\left([.,.]_{\pi}\right)_{N^*}$ and $\left([.,.]_{N}\right)_{\pi}$, where $N^*$ stands for the transpose of $N$. More precisely, the concomitant $C_{\pi, N}$ is given, for all $\alpha, \beta \in \Gamma(A^*)$, by
$$C_{\pi, N}(\alpha, \beta)= \left([\alpha, \beta]_{N}\right)_{\pi} - \left([\alpha, \beta]_{\pi}\right)_{N^*},$$
or, in terms of big bracket and functions on $\mathcal{F}$,
$$C_{\pi, N}=\bb{\pi}{\bb{N}{\mu}} + \bb{N}{\bb{\pi}{\mu}}=\mu_{N,\pi} + \mu_{\pi, N}.$$
Notice that, when $C_{\pi, N}=0$, applying the Jacobi identity we get
\begin{equation}\label{Formul_C(pi,N)=0}
    \bb{\pi}{\bb{N}{\mu}}=-\bb{N}{\bb{\pi}{\mu}}=\frac{1}{2}\bb{\bb{\pi}{N}}{\mu}.
\end{equation}

\subsection{Schouten-Nijenhuis bracket and Fr\"{o}licher-Nijenhuis bracket}  

Let $(A, \mu)$ be a Lie algebroid. The Schouten-Nijenhuis bracket is the natural extension, by derivation, of the Lie bracket on $\Gamma(A)$ to a bracket on $\Gamma(\bigwedge^{\bullet} A)$. Let $P \in \Gamma(\bigwedge^{p} A)$ and $Q \in \Gamma(\bigwedge^{q} A)$ be two multivectors on $A$. The Schouten-Nijenhuis bracket of $P$ and $Q$ is the $(p+q-1)$-vector on $A$ defined, in terms of big bracket, by
\begin{equation}\label{def_SNbrk_in_bb}
    \SNbrk{P}{Q}=\bb{\bb{P}{\mu}}{Q}.
\end{equation}
Notice that, in supergeometric terms, the Lie bracket $[.,.]$ and the Schouten-Nijenhuis bracket $\SNbrk..$ are defined by the same\footnote{This was predictable because, for all $\Theta \in \mathcal{F}^{a,b}$, the operator $\bb{\Theta}{.}:\mathcal{F}\to \mathcal{F}$ is a derivation of degree $a+b$ of the algebra $(\mathcal{F}, \wedge)$.} derived bracket expression (compare (\ref{def_SNbrk_in_bb}) and the second equation in (\ref{def_rho_and_brk_in_bb})). Then, in order to simplify the notation and if there is no risk of confusion,
the Schouten-Nijenhuis bracket on $\Gamma(\bigwedge^{\bullet} A)$ will be denoted by $[.,.]$, as the Lie bracket on $\Gamma(A)$.

Let $K \in \Gamma(\bigwedge^k A^* \otimes A)$ and $L \in \Gamma(\bigwedge^l A^* \otimes A)$ be two $A$-valued forms. We denote by $i_L K$ the section of $\bigwedge^{k+l-1} A^* \otimes A$ defined by $$i_L K:=\alpha_L\wedge (i_{X_L}\alpha_K) \otimes X_K,$$ where $K=\alpha_K \otimes X_K \in \Gamma(\bigwedge^k A^* \otimes A)$ and $L=\alpha_L \otimes X_L  \in \Gamma(\bigwedge^l A^* \otimes A)$. The Fr\"{o}licher-Nijenhuis bracket~\cite{FN56} of $K$ and $L$, denoted by $\FNbrk{K}{L}$, is the section of $\bigwedge^{k+l} A^*\otimes A$ defined,
in terms of big bracket, by the simple expression
    \begin{equation}\label{def_FNbrk_en_grand_crochet}
        \FNbrk{K}{L} = \bb{\bb{K}{\mu}}{L} + (-1)^{k(l+1)} \bb{i_L K}{\mu}.
    \end{equation}
\

It is known that the Nijenhuis torsion of a vector bundle endomorphism $N:A \to A$ is given by $\mathcal{T}N=-\frac{1}{2}\FNbrk{N}{N}$.
Thus, using formula (\ref{def_FNbrk_en_grand_crochet}) and the notation introduced in (\ref{notation_mu_NS}), we get
$$\mathcal{T}N=\frac{1}{2}\left(\mu_{N,N}-\mu_{N^2}\right),$$
where $N^2=N \smc N=i_N N$.

\section{{\large\boldmath$\varepsilon$}-hypersymplectic structures on Lie algebroids}   

In this section we define an {\large$\boldsymbol{\varepsilon}$}-hypersymplectic structure on a Lie algebroid, generalizing the definition of hypersymplectic triple given in~\cite{Xu97} for manifolds.
As we will see, an {\large$\boldsymbol{\varepsilon}$}-hypersymplectic structure induces some other structures on the Lie algebroid. We study the main properties of these induced structures and the relations between them.

Let $A\to M$ be a vector bundle endowed with a Lie algebroid structure $(\rho, \brk..)$, corresponding to a function $\mu \in \mathcal{F}^{1,2}$.

Take three symplectic forms $\omega_1, \omega_2$ and $\omega_3 \in \Gamma(\wedge^2 A^*)$ with inverse Poisson bivectors $\pi_1, \pi_2$ and $\pi_3 \in \Gamma(\wedge^2 A)$, respectively. Then, for all $i \in \{1,2,3\}$ we have
$${\omega_i}^{\flat}\smc{\pi_i}^{\sharp}={\rm Id}_{A^*}\quad\text{and}\quad {\pi_i}^{\sharp}\smc{\omega_i}^{\flat}={\rm Id}_{A},$$
where $\omega_i^\flat: \Gamma(A)\to \Gamma(A^*)$ is defined by setting $\langle\omega_i^\flat(X),Y\rangle:=\omega_i(X,Y)$, for all $X,Y \in \Gamma(A)$.

Let us define the \textit{transition $(1,1)$-tensors} $N_1,N_2,N_3: \Gamma(A) \to \Gamma(A)$, by setting
\begin{equation}\label{def_Ni}
  N_i:=\pi_{i-1}^{\sharp}\smc\omega_{i+1}^{\flat},
\end{equation}
considering the indices in $\mathbb{Z}_3$.

\begin{rem}
We consider $1, 2$ and $3$ as the representative elements of the equivalence classes of $\mathbb{Z}_3$, i.e., $\mathbb{Z}_3:=\{[1],[2],[3]\}$. In what follows, although we omit the brackets, and write $i$ instead of $[i]$, all the indices (and corresponding computations) must be thought in $\mathbb{Z}_3$.
\end{rem}

\begin{defn}
    The triple $(\omega_1, \omega_2, \omega_3)$ is an \textit{{\large${\boldsymbol\varepsilon}$}-hypersymplectic structure} on a Lie algebroid $(A, \rho, [\cdot,\cdot])$ if the transition $(1,1)$-tensors satisfy
    \begin{equation}\label{eq_Ij_almostCPS}
        {N_i}^2=\varepsilon_i {\rm Id}_{A}, \quad\quad i=1,2,3,
    \end{equation}
        where the parameters $\varepsilon_i=\pm 1$ form the triple {\large${\boldsymbol\varepsilon}$}$=(\varepsilon_1,\varepsilon_2,\varepsilon_3)$.
\end{defn}

\begin{rem}
    Condition (\ref{eq_Ij_almostCPS}) can be written using only the symplectic forms and their inverse Poisson bivectors as, for example, in the following formula where the indices are treated as belonging to $\mathbb{Z}_3$,
    \begin{equation}
        \omega_{i+1}^{\flat}\smc\pi_{i-1}^{\sharp}= \varepsilon_i\, \omega_{i-1}^{\flat}\smc\pi_{i+1}^{\sharp}. \label{Formul_Ni2=EId_in_omega}
    \end{equation}
\end{rem}

\begin{prop}\label{prop_algebra_Ni}
Let $(\omega_1, \omega_2, \omega_3)$ be an {$\Large\varepsilon$}-hypersymplectic structure on $(A, \mu)$. The transition $(1,1)$-tensors satisfy, for all $i \in \{1,2,3\}$,
        \begin{enumerate}
            \item[(a)] $(N_i)^{-1}=\varepsilon_i N_i$,
            \item[(b)] $N_3 N_2 N_1 = {\rm{Id_A}}$ and $N_1 N_2 N_3 = \varepsilon_1\varepsilon_2\varepsilon_3\ {\rm{Id_A}}$.
        \end{enumerate}
\end{prop}

Notice that, using the symplectic forms and their inverse Poisson bivectors, the condition (a) in Proposition~\ref{prop_algebra_Ni} is given by
    \begin{equation}
        \pi_{i+1}^{\sharp}\smc\omega_{i-1}^{\flat}= \varepsilon_i\, \pi_{i-1}^{\sharp}\smc\omega_{i+1}^{\flat}. \label{Formul_Ni2=EId_in_omega2}
    \end{equation}

\begin{defn}
    We define $g \in \bigotimes^2 A^*$ by setting, for all $X,Y \in \Gamma(A)$, $$g(X,Y):=\langle g^\flat X, Y\rangle,$$ where $g^\flat: A\longrightarrow A^*$ is given by
    \begin{equation}\label{first_defn _g}
    g^\flat:=\varepsilon_3\varepsilon_2\ {\omega_3}^{\flat} \smc {\pi_1}^{\sharp} \smc {\omega_2}^{\flat}.
    \end{equation}
\end{defn}

In the next proposition we show that the definition of $g^\flat$ is not affected by a circular permutation of the indices in equation~(\ref{first_defn _g}) and that $g$ is symmetric or skew-symmetric, depending on the sign of $\varepsilon_1\varepsilon_2\varepsilon_3$.

\begin{prop}\label{Prop_properties_g}
    Let $(\omega_1, \omega_2, \omega_3)$ be an {$\Large\varepsilon$}-hypersymplectic structure on $(A, \mu)$. The morphism $g^\flat$ defined in (\ref{first_defn _g}) satisfies:
    \begin{enumerate}
      \item[(a)] $g^\flat=\varepsilon_{i-1}\varepsilon_{i+1}\ {\omega_{i-1}}^{\flat} \smc {\pi_i}^{\sharp} \smc {\omega_{i+1}}^{\flat}$, \quad for all $i \in \mathbb{Z}_3$;
      \item[(b)] $(g^\flat)^* = - \varepsilon_1\varepsilon_2\varepsilon_3\ g^\flat$.
    \end{enumerate}
\end{prop}
\begin{proof}
    \begin{enumerate}
      \item[(a)] 
          Starting from the definition of $g^\flat$, applying formula (\ref{Formul_Ni2=EId_in_omega2}) and then formula (\ref{Formul_Ni2=EId_in_omega}), we have:
          $$g^\flat=\varepsilon_3\varepsilon_2\ {\omega_3}^{\flat} \smc {\pi_1}^{\sharp} \smc {\omega_2}^{\flat}
          =\varepsilon_2\ {\omega_3}^{\flat} \smc {\pi_2}^{\sharp} \smc {\omega_1}^{\flat}
          =\varepsilon_1\varepsilon_2\ {\omega_2}^{\flat} \smc {\pi_3}^{\sharp} \smc {\omega_1}^{\flat}.$$
          This last equality is obtained from (\ref{first_defn _g}), after a circular permutation of the indices. We can continue permuting the indices, using formulae (\ref{Formul_Ni2=EId_in_omega2}) and (\ref{Formul_Ni2=EId_in_omega}), to get:
          $$g^\flat=\varepsilon_1\varepsilon_2\ {\omega_2}^{\flat} \smc {\pi_3}^{\sharp} \smc {\omega_1}^{\flat}
          =\varepsilon_1\ {\omega_2}^{\flat} \smc {\pi_1}^{\sharp} \smc {\omega_3}^{\flat}
          =\varepsilon_1\varepsilon_2\ {\omega_1}^{\flat} \smc {\pi_2}^{\sharp} \smc {\omega_3}^{\flat}.$$
          Therefore, for any $i \in \mathbb{Z}_3$,  $g^\flat=\varepsilon_{i-1}\varepsilon_{i+1}\ {\omega_{i-1}}^{\flat} \smc {\pi_i}^{\sharp} \smc {\omega_{i+1}}^{\flat}$.
      \item[(b)] Using (a) and the fact that both $\omega_i$ and $\pi_i$, $i=1,2,3$, are skew-symmetric, we obtain
      \begin{align*}
        (g^\flat)^*&=-\varepsilon_{i-1}\varepsilon_{i+1}\ {\omega_{i+1}}^{\flat} \smc {\pi_i}^{\sharp} \smc {\omega_{i-1}}^{\flat}\\
        &=-\varepsilon_{i-1}\ {\omega_{i+1}}^{\flat} \smc {\pi_{i-1}}^{\sharp} \smc {\omega_{i}}^{\flat},
      \end{align*}
      where we used formula (\ref{Formul_Ni2=EId_in_omega2}). Since, from (a), $g^\flat=\varepsilon_{i+1}\varepsilon_{i}\ {\omega_{i+1}}^{\flat} \smc {\pi_{i-1}}^{\sharp} \smc {\omega_{i}}^{\flat}$, we get
      $$(g^\flat)^*= - \varepsilon_1\varepsilon_2\varepsilon_3\ g^\flat.$$
    \end{enumerate}
\end{proof}

We define the inverse of $g$, $g^{-1} \in \bigotimes^2 A$, by setting
$$g^{-1}(\alpha,\beta):=\langle(g^\flat)^{-1}(\alpha), \beta\rangle,$$
for all $\alpha,\beta \in \Gamma(A^*)$. As a direct consequence of this definition, we have \mbox{$(g^{-1})^{\sharp}=(g^\flat)^{-1}$}.

\begin{prop}\label{Prop_Relation_morphisms_i=j}
    The morphisms $g^\flat, {\omega_i}^{\flat}$ and $N_i$ satisfy the following:
    \begin{enumerate}
      \item[(a)] $\omega_i^{\flat} \smc N_i = \varepsilon_1\varepsilon_2\varepsilon_3 {N_i}^* \smc \omega_i^{\flat}= \varepsilon_{i-1} g^\flat$,
      \item[(b)] $\pi_i^{\sharp} \smc {N_i}^* = \varepsilon_1\varepsilon_2\varepsilon_3 N_i \smc \pi_i^{\sharp}= \varepsilon_{i+1} (g^{-1})^{\sharp}$,
      \item[(c)] $g^\flat \smc N_i = \varepsilon_1\varepsilon_2\varepsilon_3 {N_i}^* \smc g^\flat= \varepsilon_{i}\varepsilon_{i-1} \omega_i^{\flat}$,
    \end{enumerate}
    for all indices in $\mathbb{Z}_3$.
\end{prop}
\begin{proof}
    \begin{enumerate}
      \item[(a)] Using (\ref{def_Ni}), (\ref{Formul_Ni2=EId_in_omega2}) and Proposition~\ref{prop_algebra_Ni}, we have
            $$\omega_i^{\flat} \smc N_i = \omega_i \smc \pi_{i-1}^{\sharp}\smc\omega_{i+1}^{\flat}
            = \varepsilon_i \omega_i \smc \pi_{i+1}^{\sharp}\smc\omega_{i-1}^{\flat}
            = \varepsilon_{i-1} g^{\flat}.$$
      The remaining part of the statement comes from the fact that $(g^\flat)^* = - \varepsilon_1\varepsilon_2\varepsilon_3\ g^\flat$. Indeed,
      $$\omega_i^{\flat} \smc N_i
      =\varepsilon_{i-1}g^{\flat}
      =-\varepsilon_{i}\varepsilon_{i+1}(g^{\flat})^*
      =-\varepsilon_{i}\varepsilon_{i+1}(\varepsilon_{i-1}\omega_i^{\flat} \smc N_i)^*
      = \varepsilon_1\varepsilon_2\varepsilon_3\ {N_i}^* \smc \omega_i^{\flat}.$$

      \

      \item[] Statements (b) and (c) can be proved directly, by the same kind of arguments, or using (a).
    \end{enumerate}
\end{proof}
\begin{prop}\label{prop_hermicity_(g,Ni)}
    For all sections $X, Y \in \Gamma(A)$ and for all indices in $\mathbb{Z}_3$, $$g(N_i X, N_i Y) = \varepsilon_{i-1}\varepsilon_{i+1}\ g(X,Y).$$
\end{prop}
\begin{proof}
    Let $X$ and $Y$ be sections of $A$. We have
    \begin{align*}
        g(N_i X, N_i Y) &= \langle g^{\flat}(N_i X), N_i Y\rangle
        = \varepsilon_1\varepsilon_2\varepsilon_3 \langle {N_i}^*(g^{\flat} X), N_i Y\rangle\\
        &= \varepsilon_1\varepsilon_2\varepsilon_3 \langle g^{\flat} X, (N_i)^2(Y)\rangle
        = \varepsilon_{i-1}\varepsilon_{i+1} \langle g^{\flat} X, Y\rangle\\
        &= \varepsilon_{i-1}\varepsilon_{i+1}\ g(X,Y),
    \end{align*}
    where we used Proposition~\ref{Prop_Relation_morphisms_i=j}(c) in the second equality.
\end{proof}
The next proposition is the continuation of Proposition \ref{Prop_Relation_morphisms_i=j}, in the cases where the indices are different.

\begin{prop}\label{Prop_Relation_morphisms_i<>j}
The morphisms induced by an {$\Large\varepsilon$}-hypersymplectic structure $(\omega_1, \omega_2, \omega_3)$ satisfy the following relations, for all indices $i, k \in \mathbb{Z}_3, i\neq k$,
\begin{enumerate}
\item[(a)] $\omega_i^{\flat} \smc N_k = {N_k}^* \smc \omega_i^{\flat} =\left\{
                                            \begin{array}{ll}
                                                   \omega_{i-1}^{\flat} & ,\ k=i+1 \\
                                                   \varepsilon_{i-1}\omega_{i+1}^{\flat} &  ,\ k=i-1;
                                            \end{array}
                                           \right.$

                                           \

\item[(b)] $\pi_i^{\sharp} \smc {N_k}^* = N_k \smc \pi_i^{\sharp} =\left\{
                                            \begin{array}{ll}
                                                   \varepsilon_{i+1}\pi_{i-1}^{\sharp} & ,\ k=i+1 \\
                                                   \pi_{i+1}^{\sharp} &  ,\ k=i-1;
                                            \end{array}
                                           \right.$

                                           \

\item[(c)] $N_i \smc N_k = \varepsilon_1\varepsilon_2\varepsilon_3 N_k \smc N_i =\left\{
                                            \begin{array}{ll}
                                                   \varepsilon_{i}\varepsilon_{i+1}N_{i-1} & ,\ k=i+1 \\
                                                   \varepsilon_{i+1}N_{i+1} &  ,\ k=i-1.
                                            \end{array}
                                           \right.$
\end{enumerate}
\end{prop}
\begin{proof}
    The proofs of all statements are done by direct computations. We prove (a) to illustrate the kind of computations we use.
    \begin{enumerate}
        \item[(a)] If $k=i+1$,
        $$\omega_i^{\flat} \smc N_k=\omega_i^{\flat} \smc N_{i+1}=\omega_i^{\flat} \smc \pi_i^{\sharp} \smc \omega_{i-1}^{\flat}=\omega_{i-1}^{\flat}.$$
        On the other hand, if $k=i-1$,
        \begin{align*}
            \omega_i^{\flat} \smc N_k&=\omega_i^{\flat} \smc N_{i-1}=\omega_i^{\flat} \smc \pi_{i+1}^{\sharp} \smc  \omega_{i}^{\flat}\\
            &=\varepsilon_{i-1} \omega_i^{\flat} \smc \pi_i^{\sharp} \smc  \omega_{i+1}^{\flat}=\varepsilon_{i-1} \omega_{i+1}^{\flat},
        \end{align*}
        where we used (\ref{Formul_Ni2=EId_in_omega2}) in the third equality.
        In both cases, $\omega_i^{\flat} \smc N_k$ is skew-symmetric, so that $\omega_i^{\flat} \smc N_k=-(\omega_i^{\flat} \smc N_k)^*$, i.e.,
        $$\omega_i^{\flat} \smc N_k={N_k}^* \smc \omega_i^{\flat}.$$
    \end{enumerate}
\end{proof}

The next diagram shows all the relations between the morphisms induced by an {{\large$\boldsymbol{\varepsilon}$}-hypersymplectic structure on a Lie algebroid.

\begin{figure}[H]
\begin{center}
\psscalebox{0.7}{
\begin{pspicture*}(-2.75,-1.16)(17,11.26)
\psline[linewidth=1.6pt](3,5.2)(6,0)
\psline[linewidth=1.6pt](0,0)(6,0)
\psline[linewidth=1.6pt](0,0)(3,5.2)
\psline(0,0)(3,1.73)
\psline(3,5.2)(3,1.73)
\psline(6,0)(3,1.73)
\psline(9,5.2)(6,3.46)
\psline[linewidth=1.6pt](9,5.2)(3,5.2)
\psline(3,5.2)(6,3.46)
\psline(6,0)(6,3.46)
\psline[linewidth=1.6pt](9,5.2)(6,0)
\psline(3,5.2)(6,6.93)
\psline(9,5.2)(6,6.93)
\psline[linewidth=1.6pt](9,5.2)(6,10.39)
\psline(6,10.39)(6,6.93)
\psline[linewidth=1.6pt](3,5.2)(6,10.39)
\psline[linewidth=1.6pt](9,5.2)(12,0)
\psline(12,0)(9,1.73)
\psline[linewidth=1.6pt](12,0)(6,0)
\psline(6,0)(9,1.73)
\psline(9,5.2)(9,1.73)
\begin{huge}
\rput[tl](5.66,5.4){$\boldsymbol >$}
\rput[tl]{60}(7.37,2.8){$\boldsymbol <$}
\rput[tl]{300}(4.46,3.1){$\boldsymbol <$}
\rput[tl](2.68,0.2){$\boldsymbol <$}
\rput[tl](8.68,0.2){$\boldsymbol >$}
\rput[tl]{60}(1.245,2.58){$\boldsymbol <$}
\rput[tl]{60}(4.15,7.61){$\boldsymbol >$}
\rput[tl]{300}(10.45,3.08){$\boldsymbol >$}
\rput[tl]{300}(7.5,8.2){$\boldsymbol <$}
\end{huge}
\begin{large}
\rput[tl](5.7,4.9){$N_1 $}
\rput[tl]{60}(7,3){$N_2 $}
\rput[tl]{300}(5.0,3.2){$N_3 $}
\rput[tl](8.83,-0.25){${\omega_1}^{\flat}$}
\rput[tl](2.72,-0.25){${\omega_1}^{\flat}$}
\rput[tl]{60}(0.9,2.9){${\omega_2}^{\flat}$}
\rput[tl]{60}(3.9,7.95){${\omega_2}^{\flat}$}
\rput[tl]{300}(10.9,3.1){${\omega_3}^{\flat}$}
\rput[tl]{300}(8,8.3){${\omega_3}^{\flat}$}
\end{large}
\rput[tl]{330}(4.68,4.36){$\boldsymbol <$}
\rput[tl]{270}(6.12,2.2){$\boldsymbol >$}
\rput[tl]{30}(7.12,4.25){$\boldsymbol >$}
\rput[tl]{30}(1.65,1.08){$\boldsymbol <$}
\rput[tl]{90}(2.88,2.9){$\boldsymbol <$}
\rput[tl]{330}(4.15,1.2){$\boldsymbol <$}
\rput[tl]{30}(4.65,6.28){$\boldsymbol >$}
\rput[tl]{90}(5.88,8.1){$\boldsymbol >$}
\rput[tl]{330}(7.15,6.4){$\boldsymbol <$}
\rput[tl]{30}(7.65,1.08){$\boldsymbol >$}
\rput[tl]{90}(8.88,2.9){$\boldsymbol <$}
\rput[tl]{330}(10.15,1.2){$\boldsymbol >$}
\rput[tl]{330}(4.32,4.14){$\varepsilon_1 N_2 $}
\rput[tl]{270}(6.45,2.4){$\varepsilon_3 N_1 $}
\rput[tl]{30}(6.9,4.55){$\varepsilon_2 N_3 $}
\rput[tl]{30}(1.22,1.35){$\varepsilon_1 {N_3}^*$}
\rput[tl]{90}(3.25,2.8){$\varepsilon_3\varepsilon_1 {\omega_1}^{\flat}$}
\rput[tl]{330}(3.55,1.1){$\varepsilon_1 {\omega_2}^{\flat}$}
\rput[tl]{30}(4.22,6.55){$\varepsilon_2 {\omega_3}^{\flat}$}
\rput[tl]{90}(6.25,8){$\varepsilon_2{N_1}^*$}
\rput[tl]{330}(6.55,6.3){$\varepsilon_1\varepsilon_2 {\omega_2}^{\flat}$}
\rput[tl]{30}(7.22,1.35){$\varepsilon_2\varepsilon_3 {\omega_3}^{\flat}$}
\rput[tl]{90}(9.25,2.8){$\varepsilon_3 {\omega_1}^{\flat}$}
\rput[tl]{330}(9.55,1.1){$\varepsilon_3 {N_2}^*$}
\begin{large}
\rput(-0.5,-0.25){$\boldsymbol D$}
\rput(6,10.85){$\boldsymbol D$}
\rput(12.4,-0.25){$\boldsymbol D$}
\rput(6,-0.5){$\boldsymbol A$}
\rput(2.5,5.3){$\boldsymbol B$}
\rput(9.4,5.3){$\boldsymbol C$}
\end{large}
\end{pspicture*}
}
\end{center}
\caption{}\label{Fig. 2}
\end{figure}

This is to be understood as the pattern for a tetrahedron $ABCD$. The morphism $g^\flat$ does not appear in Figure \ref{Fig. 2} but, as shown in Figure \ref{Fig. 3}, $g^\flat$ is the altitude of the tetrahedron $ABCD$.

\begin{figure}[H]
\psscalebox{0.8}{
\begin{pspicture*}(-2,-3.5)(7,3.5)
\psline[linewidth=1.6pt,linestyle=dashed,dash=5pt 5pt](-1.12,-0.84)(4.08,-0.06)
\psline[linestyle=dashed,dash=5pt 5pt](1.82,-1.26)(-1.12,-0.84)
\psline[linestyle=dashed,dash=5pt 5pt](1.82,-1.26)(2.5,-2.88)
\psline[linestyle=dashed,dash=5pt 5pt](1.82,-1.26)(4.08,-0.06)
\psline[linewidth=1.6pt,linestyle=dashed,dash=12pt 7pt,linecolor=red](1.82,2.8)(1.82,-1.26)
\psline[linewidth=1.6pt](1.82,2.8)(-1.12,-0.84)
\psline[linewidth=1.6pt](-1.12,-0.84)(2.5,-2.88)
\psline[linewidth=1.6pt](2.5,-2.88)(1.82,2.8)
\psline[linewidth=1.6pt](1.82,2.8)(4.08,-0.06)
\psline[linewidth=1.6pt](4.08,-0.06)(2.5,-2.88)
\psline[linewidth=1.6pt](2.5,-2.88)(1.82,2.8)
\rput[tl](1.32,0.65){\large\boldmath$\red{ g^\flat }$}
\rput[tl]{90}(1.705,0.40){\boldmath$\red{\Huge >}$}
\rput[tl]{90}(1.705,0.41){\boldmath$\red{\Huge >}$}
\rput[tl]{90}(1.705,0.42){\boldmath$\red{\Huge >}$}
\rput(1.8,3.1){$\boldsymbol D$}
\rput(2.5,-3.2){$\boldsymbol A$}
\rput(-1.5,-0.8){$\boldsymbol B$}
\rput(4.4,0){$\boldsymbol C$}
\end{pspicture*}
}
\caption{}\label{Fig. 3}
\end{figure}

The next proposition is a  collection of already proved relations between ${\omega_i}^{\flat}, {\pi_i}^{\sharp}, N_k$ and $g$. The novelty is that we write these relations using the big bracket and functions on $\mathcal{F}$.

\begin{prop}\label{Prop_PHS_relations_in_BB}
    For all indices $i,k$ in $\mathbb{Z}_3$,
    \begin{enumerate}
      \item[(a)] $\bb{\omega_i}{\pi_k}=\left\{
                                            \begin{array}{ll}
                                                   \rm{Id}_A  & ,\ k=i \\
                                                   N_{i-1} & ,\ k=i+1 \\
                                                   \varepsilon_{i+1} N_{i+1} &  ,\ k=i-1;
                                            \end{array}
                                           \right.$

                                           \

      \item[(b)] $\bb{N_k}{N_{k+1}}=-\bb{N_{k+1}}{N_k}= \varepsilon_{k-1}\,(1-\varepsilon_1\varepsilon_2\varepsilon_3)\,N_{k-1}$;

      \

      \item[(c)] $\bb {N_k}{\omega_i}=\left\{\begin{array}{ll}
                                             \varepsilon_{i-1}\,(1+\varepsilon_1\varepsilon_2\varepsilon_3)\,g  & ,\ k=i \\
                                             2\, \omega_{i-1} & ,\ k=i+1 \\
                                             2\,\varepsilon_{i-1} \omega_{i+1} &  ,\ k=i-1;
                                        \end{array}
                                  \right.$

                                  \

      \item[(d)] $\bb {N_k}{\pi_i}=\left\{\begin{array}{ll}
                                           -\varepsilon_{i+1}\,(1+\varepsilon_1\varepsilon_2\varepsilon_3)\,g^{-1}  & ,\ k=i \\
                                           -2\, \varepsilon_{i+1} \pi_{i-1} & ,\ k=i+1 \\
                                           -2\, \pi_{i+1} &  ,\ k=i-1.
                                     \end{array}
                              \right.$
    \end{enumerate}
\end{prop}
\begin{proof}
    Let us prove (c) to exemplify how the already proved results translate in terms of big bracket and functions on $\mathcal{F}$. For a better understanding of these computations, see \cite{roythesis, A10}.

    (c) We can identify the function $\bb {N_k}{\omega_i} \in \mathcal{F}$ with a $2$-form on $A$ such that $\bb {N_k}{\omega_i}^\flat=\omega_i^\flat \smc N_k + N_k^* \smc \omega_i^\flat$. Then, using Proposition~\ref{Prop_Relation_morphisms_i=j}(a) and Proposition~\ref{Prop_Relation_morphisms_i<>j}(a), we have:
    \begin{align*}
        \bb {N_k}{\omega_i}^\flat&=\omega_i^\flat \smc N_k + N_k^* \smc \omega_i^\flat\\
        &=\left\{\begin{array}{ll}
                    (1+\varepsilon_1\varepsilon_2\varepsilon_3)\, \omega_i^\flat \smc N_i & ,\ k=i \\
                    2\,\omega_i^\flat \smc N_k & ,\ k\neq i
                \end{array}
        \right.\\
        &=\left\{\begin{array}{ll}
                    \varepsilon_{i-1}(1+\varepsilon_1\varepsilon_2\varepsilon_3)\, g^\flat & ,\ k=i \\
                    2\,\omega_{i-1}^\flat & ,\ k= i+1\\
                    2\,\varepsilon_{i-1}\omega_{i+1}^\flat & ,\ k= i-1.
                \end{array}
        \right.
    \end{align*}
\end{proof}
\

We recall an useful lemma from \cite{KSR10}.

\begin{lem}
    A bivector $\pi$, a $2$-form $\omega$ and a $(1,1)$-tensor $N$ on a Lie algebroid $(A, \mu)$, related by $N=\pi^\sharp\smc\omega^\flat$, satisfy the relation
    \begin{equation}\label{lema_YKS+R}
        \bb{\bb{\brk{\pi}{\pi}}{\omega}}{\omega}=\bb{\bb{\bb{\pi}{{\rm d}\omega}}{\pi}}{\omega}- \bb{\bb{\pi}{N}}{{\rm d}\omega} + 2\bb{\pi}{\bb{\omega}{\bb{N}{\mu}}} + 4 \mathcal{T}N.
    \end{equation}
\end{lem}

The next proposition shows that the transition tensors of an {\large$\boldsymbol{\varepsilon}$}-hypersymplectic structure are Nijenhuis tensors.

\begin{prop}\label{Ni_are_Nijenhuis}
    Let $(\omega_1, \omega_2, \omega_3)$ be an {\large$\boldsymbol{\varepsilon}$}-hypersymplectic structure on a Lie algebroid $(A, \mu)$. The transition $(1,1)$-tensors are Nijenhuis tensors, i.e., for all $i \in \{1,2,3\}$, $$\mathcal{T} N_i=0.$$
\end{prop}
\begin{proof}
    By definition, $N_i=\pi_{i-1}^\sharp\smc\omega_{i+1}^\flat$ and, using the fact that $\omega_{i+1}$ is closed and $\pi_{i-1}$ is a Poisson bivector, the formula (\ref{lema_YKS+R}) reduces to
    $$\mathcal{T} N_i= -\frac{1}{2} \bb{\pi_{i-1}}{\bb{\omega_{i+1}}{\bb{N_i}{\mu}}}.$$
    Now, using the Jacobi identity, Proposition \ref{Prop_PHS_relations_in_BB}(c) and the fact that $\omega_{i-1}$ is closed, we have
    $$\bb{\omega_{i+1}}{\bb{N_i}{\mu}}=\bb{\bb{\omega_{i+1}}{N_i}}{\mu}
    =\bb{-2\,\varepsilon_i \omega_{i-1}}{\mu}=0.$$
    Therefore, $$\mathcal{T} N_i=0.$$
\end{proof}

\section{Induced compatible structures}   

In this section we show that an {\large$\boldsymbol{\varepsilon}$}-hypersymplectic structure induces many pairs of compatible structures such as, amongst others, pairs of compatible Poisson bivectors and Poisson-Nijenhuis structures.

Let $(A,\mu)$ be a Lie algebroid.
Recall that a pair $(\pi, N)$, where $\pi$ is a bivector and $N$ is a  $(1,1)$-tensor on $A$ is a {\em Poisson-Nijenhuis structure} ($PN$ structure, for short) on $(A,\mu)$ if
 \begin{equation}
[ \pi, \pi ] =0, \,\,\,\, \mathcal{T} N=0, \,\,\,\, N \smc \pi^\sharp= \pi^\sharp \smc  N^*\,\,\,\, {\textrm{and}} \,\,\,\, C_{\pi, N}=0.
\end{equation}

A pair $(\omega,N)$ formed by a $2$-form $\omega$ and a $(1,1)$-tensor $N$ on $A$ is an $\Omega N${\em structure} on $(A,\mu)$ if
\begin{equation}  \label{def_Omega_N}
{\rm d} \omega=0, \,\,\,\, \mathcal{T} N=0, \,\,\,\, \omega^\flat \smc N= N^* \smc \omega^\flat \,\,\,\, {\textrm{and}} \,\,\,\, {\rm d}(\omega_N)=0,
\end{equation}
where $\omega_N( .,.)=\omega(N.,.)$ or, equivalently, $\omega_N^\flat=\omega^\flat \smc N$.

A pair $(\pi,\omega)$ formed by a bivector $\pi$ and a $2$-form $\omega$ on $A$ is a $P \Omega$ {\em structure} on $(A,\mu)$ if
\begin{equation}  \label{def_P_Omega}
[\pi, \pi]=0, \,\,\,\, {\rm d}\omega=0 \,\,\,\, {\textrm{and}} \,\,\,\, {\rm d} (\omega_{N})=0,
\end{equation}
where $N$ is the $(1,1)$-tensor on $A$ defined by $N=\pi^\sharp \smc \omega^\flat$.

\begin{prop}\label{prop_Induced_P-Omega-struct_first}
    Let $(\omega_1, \omega_2, \omega_3)$ be an {\large$\boldsymbol{\varepsilon}$}-hypersymplectic structure on a Lie algebroid $(A, \mu)$. For all $i\in \mathbb{Z}_3$,
    \begin{enumerate}
      \item[(a)] $(\pi_{i+1}, \omega_i)$ is a $P\Omega$ structure,
      \item[(b)] $(\pi_{i-1}, \omega_i)$ is a $P\Omega$ structure.
    \end{enumerate}
\end{prop}
\begin{proof}
    \begin{enumerate}
      \item[(a)] We only need to prove that
       ${\rm d}\!\left(({\omega_i})_{\pi_{i+1}\smc \omega_i}\right)=0$. Equation (\ref{def_Ni}) gives $\pi_{i+1}\smc \omega_i=N_{i-1}$ and from Proposition \ref{Prop_Relation_morphisms_i<>j} (a), we get $\omega_i^\flat \smc N_{i-1}=\varepsilon_{i-1} \omega_{i+1}^\flat$. Then, ${\rm d}\!\left(({\omega_i})_{\pi_{i+1}\smc \omega_i}\right)={\rm d}\left(\varepsilon_{i-1} \omega_{i+1}\right)=0$.
          Therefore, $(\pi_{i+1}, \omega_i)$ is a $P\Omega$ structure.
      \item[(b)] Analogously, $(\pi_{i-1}, \omega_i)$ is a $P\Omega$ structure. In fact,
      $${\rm d}\!\left(({\omega_i})_{\pi_{i-1}\smc \omega_i}\right)={\rm d} \! \left(\varepsilon_{i} ({\omega_{i}})_{N_{i+1}}\right)={\rm d}\! \left(\varepsilon_{i} \, \omega_{i-1}\right)=0.$$
    \end{enumerate}
\end{proof}

Recall the following result from \cite{KSR10, A10}:

\begin{prop}\label{prop_P-Omega_induces_PN_and_OmegaN}
    Let $(\pi, \omega)$ be a $P\Omega$ structure on a Lie algebroid $(A, \mu)$ and define $N:=\pi^\sharp \smc \omega^\flat$. Then,
    \begin{enumerate}
        \item[(a)] $(\pi, N)$ is a $PN$ structure on $A$,
        \item[(b)] $(\omega, N)$ is an $\Omega N$ structure on $A$.
    \end{enumerate}
\end{prop}

\noindent Then, Proposition~\ref{prop_Induced_P-Omega-struct_first} has an immediate corollary.

\begin{cor}\label{cor_PN, ON-structure}
    Let $(\omega_1, \omega_2, \omega_3)$ be an {\large$\boldsymbol{\varepsilon}$}-hypersymplectic structure on a Lie algebroid $(A, \mu)$. For all $i, k\in \{1,2,3\}$, with $i \neq k$,
    \begin{enumerate}
      \item[(a)] $(\pi_i, N_k)$ is a $PN$ structure,
      \item[(b)] $(\omega_i, N_k)$ is an $\Omega N$ structure.
    \end{enumerate}
\end{cor}

Notice that, in general, an {\large$\boldsymbol{\varepsilon}$}-hypersymplectic structure induces $6$ $PN$ structures. In fact, contrary to what is claimed in~\cite{CamacaroCarinena}, the pairs $(\pi_i, N_i), i=1,2,3,$ are $PN$ structures only when $\varepsilon_1\varepsilon_2\varepsilon_3=1$, as we will see in the next section.

The remaining results of this section deal with compatibility between two bivectors or two $(1,1)$-tensors induced by an {\large$\boldsymbol{\varepsilon}$}-hypersymplectic structure on a Lie algebroid $(A, \mu)$. First we recall a result on $PN$ structures.

\begin{prop}[\cite{magriYKS}]\label{prop_(pi,N)PN->pi_Poisson_for_mu_N}
    Let $(\pi, N)$ be a $PN$ structure on a Lie algebroid $(A, \mu)$. Then $\pi$ is a Poisson bivector on the Lie algebroid $(A, \mu_N)$.
\end{prop}

\begin{prop}
    Let $(\omega_1, \omega_2, \omega_3)$ be an {\large$\boldsymbol{\varepsilon}$}-hypersymplectic structure on a Lie algebroid $(A, \mu)$. For all $i, j\in \{1,2,3\}$, the Poisson bivectors $\pi_i$ and $\pi_j$ are compatible in the sense that $\pi_i + \pi_j$ is a Poisson bivector.
\end{prop}
\begin{proof}
    We prove that, for all $i \in \mathbb{Z}_3$, the Poisson bivectors $\pi_{i-1}$ and $\pi_{i+1}$ are compatible. Since they are both Poisson bivectors, it is equivalent to prove that $[\pi_{i-1},\pi_{i+1}]=0$. From Corollary~\ref{cor_PN, ON-structure}, for any $i \in \mathbb{Z}_3$, $(\pi_{i+1},N_i)$ is a $PN$ structure. Then, from Proposition~\ref{prop_(pi,N)PN->pi_Poisson_for_mu_N}, $\pi_{i+1}$ is a Poisson bivector on $(A, \mu_{N_i})$ which means that $[\pi_{i+1},\pi_{i+1}]_{N_i}=0$, or equivalently,
    $$\bb{\pi_{i+1}}{\bb{\pi_{i+1}}{\bb{N_i}{\mu}}}=0.$$
    Using (\ref{Formul_C(pi,N)=0}), we obtain
    $$\bb{\pi_{i+1}}{\frac{1}{2}\bb{\bb{\pi_{i+1}}{N_i}}{\mu}}=0.$$
    Applying Proposition~\ref{Prop_PHS_relations_in_BB}(d) we get
    $$\bb{\pi_{i+1}}{\bb{\pi_{i-1}}{\mu}}=0,$$
    which is equivalent to $[\pi_{i-1},\pi_{i+1}]=0$.
\end{proof}

\begin{thm}\label{Thm_Ni_Nk_compatible}
     Let $(\omega_1, \omega_2, \omega_3)$ be an {\large$\boldsymbol{\varepsilon}$}-hypersymplectic structure on a Lie algebroid $(A, \mu)$, such that $\varepsilon_1\varepsilon_2\varepsilon_3=-1$. The Nijenhuis tensors $N_i$ and $N_j$ are compatible, for all $i, j\in \{1,2,3\}$, in the sense that $N_i + N_j$ is a Nijenhuis tensor.
\end{thm}
\begin{proof}
    We only need to prove that $N_i$ and $N_{i+1}$ are compatible Nijenhuis tensors, for all $i \in \mathbb{Z}_3$. Since $\mathcal{T}N=-\frac{1}{2}\FNbrk{N}{N}$, for all $N \in \Gamma(A\otimes A^*)$, 
    the Nijenhuis tensors $N_i$ and $N_{i+1}$ are compatible if and only if $\FNbrk{N_i}{N_{i+1}}=0$. Using formula~(\ref{def_FNbrk_en_grand_crochet}), we have
    \begin{align}
        \FNbrk{N_i}{N_{i+1}}&=\bb{\bb{N_i}{\mu}}{N_{i+1}} +  \bb{i_{N_{i+1}} N_i}{\mu}\nonumber\\
        &=\varepsilon_{i+1}\bb{\bb{N_i}{\mu}}{\bb{\omega_i}{\pi_{i-1}}} + \varepsilon_{i}\varepsilon_{i+1}\bb{N_{i-1}}{\mu},\label{Aux4_proof_Thm_Ni_Nk_compatible}
    \end{align}
    where we used Proposition~\ref{Prop_PHS_relations_in_BB}(a) and Proposition~\ref{Prop_Relation_morphisms_i<>j}(c). If we apply the Jacobi identity in the first term of the right hand side of (\ref{Aux4_proof_Thm_Ni_Nk_compatible}), we get
    \begin{equation}\label{Aux5_proof_Thm_Ni_Nk_compatible}
        \FNbrk{N_i}{N_{i+1}}=\varepsilon_{i+1}\bb{\bb{\bb{N_i}{\mu}}{\omega_i}}{\pi_{i-1}} +\varepsilon_{i+1}\bb{\omega_i}{\bb{\bb{N_i}{\mu}}{\pi_{i-1}}} + \varepsilon_{i}\varepsilon_{i+1}\mu_{N_{i-1}}.
    \end{equation}
    Applying the Jacobi identity in the first term of the right hand side of (\ref{Aux5_proof_Thm_Ni_Nk_compatible}) and taking into account the fact that $\omega_i$ is closed, we get
    \begin{equation}\label{Aux_proof_Thm_Ni_Nk_compatible}
        \FNbrk{N_i}{N_{i+1}}=\varepsilon_{i+1}\bb{\bb{\bb{N_i}{\omega_i}}{\mu}}{\pi_{i-1}} -\varepsilon_{i+1}\bb{\omega_i}{\bb{\pi_{i-1}}{\bb{N_i}{\mu}}} + \varepsilon_{i}\varepsilon_{i+1}\mu_{N_{i-1}}.
    \end{equation}
    Let us do some computations on the second term of the right hand side of~(\ref{Aux_proof_Thm_Ni_Nk_compatible}). Using (\ref{Formul_C(pi,N)=0}), Proposition~\ref{Prop_PHS_relations_in_BB} and the closeness of $\omega_i$, we have
    \begin{align}
        \varepsilon_{i+1}\bb{\omega_i}{\bb{\pi_{i-1}}{\bb{N_i}{\mu}}}&=\varepsilon_{i+1}\bb{\omega_i}{\frac{1}{2}\bb{\bb{\pi_{i-1}}{N_i}}{\mu}}\nonumber\\
        &=\varepsilon_{i+1}\bb{\omega_i}{\bb{\varepsilon_{i}\pi_{i+1}}{\mu}}
    =\varepsilon_{i}\varepsilon_{i+1}\bb{\bb{\omega_i}{\pi_{i+1}}}{\mu}\nonumber\\
        &=\varepsilon_{i}\varepsilon_{i+1}\bb{N_{i-1}}{\mu}
        =\varepsilon_{i}\varepsilon_{i+1}\mu_{N_{i-1}}.\label{Aux2_proof_Thm_Ni_Nk_compatible}
    \end{align}
    Replacing (\ref{Aux2_proof_Thm_Ni_Nk_compatible}) in (\ref{Aux_proof_Thm_Ni_Nk_compatible}), we get
    \begin{equation}\label{Aux3_proof_Thm_Ni_Nk_compatible}
        \FNbrk{N_i}{N_{i+1}}=\varepsilon_{i+1}\bb{\bb{\bb{N_i}{\omega_i}}{\mu}}{\pi_{i-1}}.
    \end{equation}
    From Proposition~\ref{Prop_PHS_relations_in_BB}(c), since $\varepsilon_1\varepsilon_2\varepsilon_3=-1$, $\bb{N_i}{\omega_i}=0$ and the proof is complete.
\end{proof}

\section{Case $\varepsilon_1\varepsilon_2\varepsilon_3=1$: more compatible structures}   

When $\varepsilon_1\varepsilon_2\varepsilon_3=1$, from Proposition~\ref{Prop_properties_g}(b) we deduce that $g$ is a $2$-form on $A$, i.e., $g \in \Gamma(\wedge^2 A^*)$, and $g^{-1}\in \Gamma(\wedge^2 A)$ is a bivector on $A$. In the next theorem, the bivector $g^{-1}$ is induced by a $PN$ structure. Recall that, given a bivector $\pi \in \Gamma(\bigwedge^2 A)$ and a morphism $\varphi: A^* \to A^*$, we define the bivector $i_\varphi\pi$ by setting $$i_\varphi\pi(\alpha,\beta)=\pi(\varphi\alpha,\beta)-\pi(\varphi\beta,\alpha),$$ for all $\alpha, \beta \in \Gamma(A^*)$.

\begin{thm}\label{Thm_(pi_i,N_i)_PN}
     Let $(\omega_1, \omega_2, \omega_3)$ be an {\large$\boldsymbol{\varepsilon}$}-hypersymplectic structure on a Lie algebroid $(A, \mu)$, such that $\varepsilon_1\varepsilon_2\varepsilon_3=1$. The pair $(\pi_i, N_i)$ is a $PN$ structure and $i_{N_i^*}\pi_i=2\varepsilon_{i+1} g^{-1}$, for all $i \in \{1,2,3\}$.
\end{thm}

The proof of Theorem~\ref{Thm_(pi_i,N_i)_PN} uses the following lemma:
\begin{lem}\label{lem_aux_C(pi_i,N_i)=0}
    Let $(\omega_1, \omega_2, \omega_3)$ be an {\large$\boldsymbol{\varepsilon}$}-hypersymplectic structure on a Lie algebroid $(A, \mu)$. For all $i \in \mathbb{Z}_3$,
    \begin{enumerate}
        \item[(a)] $\bb{\pi_i}{\bb{N_i}{\mu}}=-\varepsilon_i\bb{N_{i+1}}{\bb{\pi_{i+1}}{\mu}}$,
        \item[(b)] $\bb{N_i}{\bb{\pi_i}{\mu}}=-\varepsilon_i\bb{\pi_{i+1}}{\bb{N_{i+1}}{\mu}}$.
    \end{enumerate}
\end{lem}
\begin{proof}
    \begin{enumerate}
        \item[(a)] From Proposition~\ref{Prop_PHS_relations_in_BB}(a), we have
        $$\bb{\pi_i}{\bb{N_i}{\mu}}=\varepsilon_i\bb{\pi_i}{\bb{\bb{\omega_{i-1}}{\pi_{i+1}}}{\mu}}.$$
        Applying twice the Jacobi identity and using the facts that $\omega_{i-1}$ is closed and $[\pi_i, \pi_{i+1}]=0$, we get
        \begin{align*}
            \bb{\pi_i}{\bb{N_i}{\mu}}&=\varepsilon_i\bb{\pi_i}{\bb{\omega_{i-1}}{\bb{\pi_{i+1}}{\mu}}}\\
            &=\varepsilon_i\bb{\bb{\pi_i}{\omega_{i-1}}}{\bb{\pi_{i+1}}{\mu}}.
        \end{align*}
        Finally, using Proposition~\ref{Prop_PHS_relations_in_BB}(a), we obtain
        $$\bb{\pi_i}{\bb{N_i}{\mu}}=-\varepsilon_i\bb{N_{i+1}}{\bb{\pi_{i+1}}{\mu}}.$$
        \item[(b)] The proof is similar to the proof of (a).
    \end{enumerate}
\end{proof}

The proof of Theorem~\ref{Thm_(pi_i,N_i)_PN} is now straightforward.

\begin{proof}
    We already know from Proposition~\ref{Prop_Relation_morphisms_i=j} that, when $\varepsilon_1\varepsilon_{2}\varepsilon_{3}=1$,
    $$\pi_i^{\sharp} \smc {N_i}^* = N_i \smc \pi_i^{\sharp}= \varepsilon_{i+1} (g^{-1})^{\sharp},$$
    so that
    $$i_{N_i^*}\pi_i=2\varepsilon_{i+1} g^{-1}.$$
    It only remains to prove that $C_{\pi_i,N_i}=0$. Using alternately Lemma~\ref{lem_aux_C(pi_i,N_i)=0}(a) and Lemma~\ref{lem_aux_C(pi_i,N_i)=0}(b), we obtain
    \begin{align*}
        \bb{\pi_i}{\bb{N_i}{\mu}}&=-\varepsilon_i\bb{N_{i+1}}{\bb{\pi_{i+1}}{\mu}}\\
        &=\varepsilon_i\varepsilon_{i+1}\bb{\pi_{i-1}}{\bb{N_{i-1}}{\mu}}\\
        &=-\varepsilon_i\varepsilon_{i+1}\varepsilon_{i-1}\bb{N_i}{\bb{\pi_i}{\mu}}.
    \end{align*}
    Since $\varepsilon_1\varepsilon_{2}\varepsilon_{3}=1$, we get
    $$\bb{\pi_i}{\bb{N_i}{\mu}}+\bb{N_i}{\bb{\pi_i}{\mu}}=0,$$
    i.e.,
    $$C_{\pi_i,N_i}=0.$$
    Therefore, $(\pi_i, N_i)$ is a $PN$ structure.
\end{proof}

Recall the following result.
\begin{prop}[\cite{magriYKS}]
    If $(\pi, N)$ is a $PN$ structure on a Lie algebroid $(A, \mu)$, then $i_{N^*}\pi$ is a Poisson bivector.
\end{prop}

As a direct consequence, because $(\pi_i, N_i)$ is a $PN$ structure with $i_{N_i^*}\pi_i=2\varepsilon_{i+1} g^{-1}$, we have:
\begin{cor}\label{cor_g_symplectic}
    The bivector $g^{-1}$ is a Poisson bivector. Equivalently, the $2$-form $g$ is symplectic.
\end{cor}

Using Corollary~\ref{cor_g_symplectic}, we can give a general version of Theorem~\ref{Thm_Ni_Nk_compatible} that includes both cases $\varepsilon_1\varepsilon_2\varepsilon_3=-1$ and $\varepsilon_1\varepsilon_2\varepsilon_3=1$.

\begin{thm}\label{Thm_Ni_Nk_compatible_complete}
     Let $(\omega_1, \omega_2, \omega_3)$ be an {\large$\boldsymbol{\varepsilon}$}-hypersymplectic structure on a Lie algebroid $(A, \mu)$. The Nijenhuis tensors $N_i$ and $N_j$ are compatible, for all $i, j\in \{1,2,3\}$, in the sense that $N_i + N_j$ is  a Nijenhuis tensor.
\end{thm}
\begin{proof}
    The case $\varepsilon_1\varepsilon_2\varepsilon_3=-1$ was treated in Theorem~\ref{Thm_Ni_Nk_compatible}, where we proved, without using any assumption on the sign of $\varepsilon_1\varepsilon_2\varepsilon_3$, equation (\ref{Aux3_proof_Thm_Ni_Nk_compatible}):
    \begin{equation*}
        \FNbrk{N_i}{N_{i+1}}=\varepsilon_{i+1}\bb{\bb{\bb{N_i}{\omega_i}}{\mu}}{\pi_{i-1}}.
    \end{equation*}

    If $\varepsilon_1\varepsilon_2\varepsilon_3=1$, we may use Proposition~\ref{Prop_PHS_relations_in_BB}(c) to get
    $$\FNbrk{N_i}{N_{i+1}}=2\varepsilon_{i+1}\varepsilon_{i-1}\bb{\bb{g}{\mu}}{\pi_{i-1}}=0,$$
    because $g$ is a closed $2$-form .
\end{proof}

When $\varepsilon_1\varepsilon_2\varepsilon_3=1$, taking into account the fact that $g$ is a symplectic form and $g^{-1}$ is a Poisson bivector, we can add $6$ new (and non-trivial) $P\Omega$ structures to the ones of Proposition~\ref{prop_Induced_P-Omega-struct_first}.

\begin{prop}\label{prop_Induced_P-Omega-struct_second}
    Let $(\omega_1, \omega_2, \omega_3)$ be an {\large$\boldsymbol{\varepsilon}$}-hypersymplectic structure on a Lie algebroid $(A, \mu)$, such that $\varepsilon_1\varepsilon_2\varepsilon_3=1$. Then, for all $i \in \{1,2,3\}$, the pairs $(\pi_i, g)$ and $(g^{-1}, \omega_i)$ are $P\Omega$ structures.
\end{prop}
\begin{proof}
    We only need to show that ${\rm d}\!\left(g_{\pi_i^\sharp\smc g^\flat}\right)=0$ and ${\rm d}\!\left(({\omega_i})_{(g^{-1})^\sharp\smc \omega_i^\flat}\right)=0$.
    From Proposition~\ref{Prop_Relation_morphisms_i=j}(a), we deduce that $\pi_i^\sharp\smc g^\flat=\varepsilon_{i-1}N_i$ and we have
    $$g_{\pi_i^\sharp\smc g^\flat}=\varepsilon_{i-1}g_{{}_{N_i}}=\varepsilon_{i}\omega_i,$$
    where we used Proposition~\ref{Prop_Relation_morphisms_i=j}(c) in the last equality. The form $\omega_i$ being closed, we conclude that
    $${\rm d}\!\left(g_{\pi_i^\sharp\smc g^\flat}\right)=0.$$
    In a similar way, we have
    $${\rm d}\!\left(({\omega_i})_{(g^{-1})^\sharp\smc \omega_i^\flat}\right)={\rm d}\! \left(\varepsilon_{i+1}({\omega_i})_{{}_{N_i}}\right)={\rm d}\!\left(\varepsilon_{i+1}\varepsilon_{i-1}g\right)=0.$$

\end{proof}

From Proposition~\ref{prop_P-Omega_induces_PN_and_OmegaN}, the $6$ $P\Omega$ structures of Proposition~\ref{prop_Induced_P-Omega-struct_second} induce $6$ $PN$ structures and $6$ $\Omega N$ structures as stated in the following two corollaries. Notice that for $3$ of these induced $PN$ structures we proved it directly in Theorem~\ref{Thm_(pi_i,N_i)_PN}.

\begin{cor}
    For all $i \in \{1,2,3\}$, the pairs $(\pi_i, N_i)$ and $(g^{-1}, N_i)$ are $PN$ structures.
\end{cor}

\begin{cor}
    For all $i \in \{1,2,3\}$, the pairs $(\omega_i, N_i)$ and $(g, N_i)$ are $\Omega N$ structures.
\end{cor}

In the last result of this section we prove that, when $\varepsilon_1\varepsilon_2\varepsilon_3=1$, the new Poisson bivector, $g^{-1}$, is compatible with any $\pi_i$, $i=1,2,3$.

\begin{prop}
    Let $(\omega_1, \omega_2, \omega_3)$ be an {\large$\boldsymbol{\varepsilon}$}-hypersymplectic structure on a Lie algebroid $(A, \mu)$, such that $\varepsilon_1\varepsilon_2\varepsilon_3=1$. For all $i \in \{1,2,3\}$, the Poisson bivectors $\pi_i$ and $g^{-1}$ are compatible.
\end{prop}
\begin{proof}
    Using Proposition~\ref{Prop_PHS_relations_in_BB}(d), we have
    $$[\pi_i, g^{-1}]=\bb{\bb{\pi_i}{\mu}}{g^{-1}}=-\frac{1}{2}\,\varepsilon_{i+1}\bb{\bb{\pi_i}{\mu}}{\bb{N_i}{\pi_i}}.$$
    Applying the Jacobi identity and using the fact that $\pi_i$ is a Poisson bivector, i.e., $\bb{\bb{\pi_i}{\mu}}{\pi_i}=0$, we get
    \begin{align*}
        [\pi_i, g^{-1}]&=-\frac{1}{2}\,\varepsilon_{i+1}\left(\bb{\bb{\bb{\pi_i}{\mu}}{N_i}}{\pi_i}+\bb{N_i}{\bb{\bb{\pi_i}{\mu}}{\pi_i}}\right)\\
        &=\frac{1}{2}\,\varepsilon_{i+1}\bb{\bb{N_i}{\bb{\pi_i}{\mu}}}{\pi_i}.
    \end{align*}
    Because $C_{\pi_i, N_i}=0$ we have $\bb{N_i}{\bb{\pi_i}{\mu}}=-\bb{\pi_i}{\bb{N_i}{\mu}}$ and so the last equality becomes
  $$ [\pi_i, g^{-1}]=-\frac{1}{2}\,\varepsilon_{i+1}\bb{\bb{\pi_i}{\bb{N_i}{\mu}}}{\pi_i}
        =-\frac{1}{2}\,\varepsilon_{i+1}[\pi_i, \pi_i]_{N_i}
        =0,$$
    where we used, in the last equality, Proposition~\ref{prop_(pi,N)PN->pi_Poisson_for_mu_N} and the fact that $(\pi_i, N_i)$ is a $PN$ structure.
Therefore, $\pi_i$ and $g^{-1}$ are compatible Poisson bivectors.
\end{proof}

\section{Case $\varepsilon_1\varepsilon_2\varepsilon_3=-1$: (para)-hyperk\"{a}hler structures}   

In this section we consider {\large$\boldsymbol{\varepsilon}$}-hypersymplectic structures with $\varepsilon_1\varepsilon_2\varepsilon_3=-1$. As we will see, these structures are in $1$-$1$ correspondence with (para-)hyperk\"{a}hler structures, a notion we will define later. First, let us consider two different cases of an {\large$\boldsymbol{\varepsilon}$}-hypersymplectic structure with $\varepsilon_1\varepsilon_2\varepsilon_3=-1$.

\begin{defn}
    Let $(\omega_1, \omega_2, \omega_3)$ be an {\large$\boldsymbol{\varepsilon}$}-hypersymplectic structure on a Lie algebroid $(A, \mu)$, such that $\varepsilon_1\varepsilon_2\varepsilon_3=-1$.
    \begin{itemize}
        \item If $\varepsilon_1=\varepsilon_2=\varepsilon_3=-1$, then $(\omega_1, \omega_2, \omega_3)$ is said to be a \emph{hypersymplectic} structure on $A$.
        \item Otherwise, $(\omega_1, \omega_2, \omega_3)$ is said to be a \emph{para-hypersymplectic} structure on $A$.
    \end{itemize}
\end{defn}

It is clear that all para-hypersymplectic structures satisfy, eventually after a cyclic permutation of the indices, $\varepsilon_1=\varepsilon_2=1$ and $\varepsilon_3=-1$. In the sequel, every para-hypersymplectic structures will be considered in such form.

\begin{defn}
    Let $(A, \mu)$ be a Lie algebroid. A \emph{pseudo-metric} on $A\to M$ is a symmetric and non-degenerate $C^\infty(M)$-linear map $\mathtt{g}: \Gamma(A)\to \Gamma(A^*)$. Furthermore, if $\mathtt{g}$ is \emph{positive definite}, i.e., if $g$ satisfies $$\langle \mathtt{g}(X),X\rangle > 0,$$ for all non vanishing sections $X \in \Gamma(A)$, then 
    $\mathtt{g}$ is a \emph{metric} on $A$.
\end{defn}

Because $\varepsilon_1\varepsilon_2\varepsilon_3=-1$, the proof of the next result is a consequence of Proposition~\ref{Prop_properties_g}.

\begin{prop}\label{prop_g_is_a_metric}
    Let $(\omega_1, \omega_2, \omega_3)$ be a (para-)hypersymplectic structure on a Lie algebroid $(A, \mu)$. The morphism $g^\flat$ defined in (\ref{first_defn _g}) is a pseudo-metric on $A$.
\end{prop}

\begin{note}
    In the sequel, we do not require the metric to be positive definite. However, in order to simplify the terminology we will omit the ``pseudo'' prefix.
\end{note}

\begin{defn}
A \emph{(para-)hermitian structure} is a pair $(\mathtt{g}, I)$  where $\mathtt{g}$ is a metric and $I$ is a (para-)complex tensor such that, for all $X,Y \in \Gamma(A)$,
$$\langle \mathtt{g}(I X), I Y \rangle = - \langle\mathtt{g}(X),I^2(Y)\rangle.$$
\end{defn}
\noindent Notice that the condition above may be written simply as $\mathtt{g}\smc I + I^*\smc \mathtt{g}=0$.

\begin{defn}
    A \emph{(para-)hyperk\"{a}hler structure} on a Lie algebroid $(A, \mu)$ is a quadruple $(\mathtt{g}, I_1, I_2, I_3)$, where the pairs $(\mathtt{g}, I_j)_{j=1,2}$ are both (para-)hermitian structures on $A$, the morphisms $I_1$ and $I_2$ anti-commute, $I_3 = I_1 I_2$, and ${\omega_i}^{\flat}:=\mathtt{g}\smc I_i$ are closed $2$-forms, for $i=1,2,3$.
\end{defn}

The main result of this section is a direct consequence of Proposition~\ref{prop_g_is_a_metric}, Proposition~\ref{prop_hermicity_(g,Ni)} and Proposition~\ref{Prop_Relation_morphisms_i<>j}.

\begin{thm}\label{Thm_1-1_correspondence}
    Let $(A, \mu)$ be a Lie algebroid. The triple $(\omega_1, \omega_2, \omega_3)$ is a (para-)hypersymplectic structure on $A$ if and only if $(g, N_1, N_2, N_3)$ is a (para-)hyperk\"{a}hler structure on $A$.
\end{thm}

\begin{note}
    The $1$-$1$ correspondence in Theorem~\ref{Thm_1-1_correspondence} concerns pseudo (para-)hyperk\"{a}hler structures. We may restrict ourselves to the more usual (para-)hyperk\"{a}hler structures $(g, N_1, N_2, N_3)$, with a positive definite metric $g$. These are in $1$-$1$ correspondence with (para-)hypersymplectic structures $(\omega_1, \omega_2, \omega_3)$, such that $g$ (defined by (\ref{first_defn _g})) is positive definite.
\end{note}


To conclude, we address a simple example in $T\mathbb{R}^4$ which provide many (para-)hypersymplectic structures.

\begin{ex}
Consider the coordinates $(x,y,p,q)$ on $\mathbb{R}^4$ and the following six $2$-forms
\begin{align*}
&\omega_1={\rm d}x \wedge {\rm d}p + {\rm d}y \wedge {\rm d}q;&&\omega_4={\rm d}x \wedge {\rm d}p - {\rm d}y \wedge {\rm d}q;\\
&\omega_2={\rm d}x \wedge {\rm d}q + {\rm d}p \wedge {\rm d}y;&&\omega_5={\rm d}x \wedge {\rm d}q - {\rm d}p \wedge {\rm d}y;\\
&\omega_3={\rm d}x \wedge {\rm d}y - {\rm d}p \wedge {\rm d}q;&&\omega_6={\rm d}x \wedge {\rm d}y + {\rm d}p \wedge {\rm d}q.
\end{align*}

These $2$-forms on $\mathbb{R}^4$ are symplectic and form a basis of the vector space of sections $\Gamma(\bigwedge^2(T^*\mathbb{R}^4))$.

For all pairwise different indices $i,j,k \in\{1,\ldots, 6\}$, the triple $(\omega_i, \omega_j, \omega_k)$ is a (para-)hypersymplectic structure on the Lie algebroid $T\mathbb{R}^4$. More precisely:
    \begin{enumerate}
        \item The triples $(\omega_1, \omega_2, \omega_3)$ and $(\omega_4, \omega_5, \omega_6)$ are hypersymplectic structures.
        \item The $9$ triples $(\omega_i, \omega_j, \omega_k)$ where $1\leq i < j \leq 3$ and $k \in \{4,5,6\}$ are para-hypersymplectic structures.
        \item The $9$ triples $(\omega_i, \omega_j, \omega_k)$ where $4\leq i < j \leq 6$ and $k \in \{1,2,3\}$ are para-hypersymplectic structures.
    \end{enumerate}

\end{ex}

\

\noindent {\bf Acknowledgments.} This work was partially supported by CMUC-FCT (Portugal) and FCT grants PEst-C/MAT/UI0324/2011 and PTDC/MAT/099880/2008 through
European program COMPETE/FEDER.


\end{document}